\documentclass[11pt]{amsart}
\usepackage{graphicx,amssymb,amsmath,amsfonts,amscd,hyperref}
\newtheorem{thm}{Theorem}[section]

\newtheorem{lem}[thm]{Lemma}

\newtheorem{rem}[thm]{Remark}

\newcommand{\Fq}{\mathbb F_q}
\newcommand{\Fqr}{\mathbb F_{q^r}}
\newcommand{\CC}{\mathbb C}

\newcommand{\QQ}{\bar{\mathbb Q}_\ell}
\newcommand{\R}{\mathrm R}

\newcommand{\Gal}{\mathrm {Gal}}
\newcommand{\AAA}{\mathbb A}
\newcommand{\PP}{\mathbb P}
\newcommand{\FF}{\mathcal F}
\newcommand{\GG}{\mathbb G}

\newcommand{\HH}{\mathrm H}
\newcommand{\HHH}{{\mathcal H}}

\newcommand{\Tr}{\mathrm {Tr}}

\newcommand{\Dbc}{{\mathcal D}^b_c}

\title[On the number of rational points on curves over finite fields]{On the number of rational points on curves over finite fields with many automorphisms}
\author{Antonio Rojas-Le\'on}
\address{Departamanto de \'Algebra,
Universidad de Sevilla, Apdo 1160, 41080 Sevilla, Spain}
\address{E-mail: arojas@us.es}
\thanks{Partially supported by P08-FQM-03894 (Junta de Andaluc\'{\i}a), MTM2007-66929 and FEDER}

\begin{document}

\maketitle

\begin{abstract}
Using Weil descent, we give bounds for the number of rational points on two families of curves over finite fields with a large abelian group of automorphisms: Artin-Schreier curves of the form $y^q-y=f(x)$ with $f\in\Fqr[x]$, on which the additive group $\Fq$ acts, and Kummer curves of the form $y^{\frac{q-1}{e}}=f(x)$, which have an action of the multiplicative group $\Fq^\star$. In both cases we can remove a $\sqrt{q}$ factor from the Weil bound when $q$ is sufficiently large.
\end{abstract}

\section{Introduction}

Let $k=\Fq$ be a finite field of characteristic $p$ and $C$ a geometrically connected smooth curve of genus $g$ in $\PP^2_k$. The well known Weil bound gives the following estimate for the number of points $N_r$ of $C$ rational over $\Fqr$ for every $r\geq 1$:

$$
|N_r-q^r-1|\leq 2gq^{\frac{r}{2}}
$$

This bound is sharp in general if we fix $C$ and take variable $\Fqr$, in the sense that
$$
\limsup_{r\geq 1}\log_q|N_r-q^r-1|=\frac{r}{2}
$$
and
$$
\limsup_{r\geq 1}\frac{|N_r-q^r-1|}{q^{\frac{r}{2}}}=2g.
$$
However, for some curves it is possible to improve this bound for large values of $q$ if we keep $r$ under control. In the article \cite{artin-schreier} it was proven that this was the case for the affine Artin-Schreier curve $A_f$ defined by
$$
y^q-y=f(x)
$$
with $f\in k[x]$, whose singular model has genus $(d-1)(q-1)/2$ and only one point at infinity. For $A_f$ one can get an estimate of the form
$$
|N_r-q^r|\leq C_{d,r}q^{\frac{r+1}{2}}
$$
under certain generic conditions on $f$ (where $N_r$ is now the number of points on the \emph{affine} curve $A_f$). In this formula $C_{d,r}$ is independent of $q$ (more precisely, it is a polynomial in $d$ of degree $r$) so it gives a great improvement of the Weil bound if $q$ is large. 
 This estimate was obtained by writing $N_r-q^r$ as a sum of additive character sums
$$
N_r-q^r=\sum_{t\in k}\sum_{x\in k^r}\psi(\mathrm{Tr}_{k_r/k}(f(x))),
$$
each of them bounded by $(d-1)q^{\frac{r}{2}}$, and then showing that there is some cancellation on the outer sum so that the total sum is bounded by $O_q(q^{\frac{r+1}{2}})$.

In this article we take a different approach: since
$$
N_f=q\cdot\#\{x\in k_r|\mathrm{Tr}_{k_r/k}(f(x))=0\},
$$
using Weil descent we construct a hypersurface in $\AAA^r_k$ whose number of rational points is precisely the number of $x\in k_r$ such that $\mathrm{Tr}_{k_r/k}(f(x))=0$. Under certain conditions the projective closure of this hypersurface is smooth, so we can use Deligne's bound to estimate its number of rational points and deduce the bound
\begin{equation}\label{mainthm1}
|N_r-q^r|\leq (d-1)^r q^{\frac{r+1}{2}}.
\end{equation}

This method is certainly less powerful than the one used in \cite{artin-schreier}. In particular, the hypotheses we need $f$ to satisfy in order to get (\ref{mainthm1}) are more restrictive than those in \cite[Corollary 3.4, Corollary 4.2]{artin-schreier}, and the constant $(d-1)^r$ is also slightly worse (notice that the coefficient of the leading term of $C_{d,r}$ in \cite[Corollary 3.4]{artin-schreier} decreases rapidly as $r$ grows). On the other hand, this method works even when $f$ is defined only over $k_r$, not just over $k$, thus giving a positive answer to one of the questions posed in the introduction of \cite{artin-schreier}.

We also apply the same procedure to study the other example proposed in the introduction of \cite{artin-schreier}: Kummer curves, a particular type of superelliptic curves of the form
$$
E_f:y^{\frac{q-1}{e}}=f(x)
$$
where $e$ is a positive divisor of $q-1$. These curves can have genus anywhere between $\left(\frac{q-1}{e}-1\right)(d-2)/2$ and $\left(\frac{q-1}{e}-1\right)(d-1)/2$, but in any case for fixed $e$ the Weil estimate gives
$$
|N_r-q^r|=O(q^{\frac{r}{2}+1}).
$$ 

Here we also have a large abelian group acting faithfully on $E_f$, namely the multiplicative group $k^\star/\mu_e$ of non-zero elements of $k$ modulo the subgroup of $e$-th roots of unity, so one also expects to be able to remove a $\sqrt{q}$ factor from the bound. We can write
$$
N_f=\delta+\frac{q-1}{e}\sum_{\lambda^e=1}\#\{x\in k_r|\mathrm{N}_{k_r/k}(f(x))=\lambda\}
$$
where $\delta$ is the number of roots of $f$ in $k_r$. Again using Weil descent we will construct a hypersurface (or rather a one-parameter family of hypersurfaces) $W_\lambda$ in $\AAA^r_k$ such that the number of rational points of $W_\lambda$ over $k$ is $\#\{x\in k_r|\mathrm{N}_{k_r/k}(f(x))=\lambda\}$. The hypersurfaces $W_\lambda$ are highly singular at infinity, so this case requires a detailed study of the cohomology of this family, which takes most of the length of this article.

The descent method works surprisingly well in this case, and we get the estimate
 $$
|N_f-q^r-\delta+1|\leq r(d-1)^{r}(q-1)q^{\frac{s-1}{2}}
$$
under the only hypothesis that $f$ is square-free of degree prime to $p$.

The fact that the descent method works well in the Kummer case and not so well in the Artin-Schreier case has an explanation: for Artin-Schreier curves, we can write $N_r-q^r$ as a ``sum of additive character sums'', parameterized by the set of non-trivial additive characters of $k$. Upon choosing a non-trivial character $\psi$, this set can be identified with the set of $k$-points of the scheme $\GG_m=\AAA^1-\{0\}$, and the corresponding exponential sums are the local Frobenius traces of the $r$-th Adams power of some geometrically semisimple $\ell$-adic sheaf on $\GG_m$. In order to get a good estimate (i.e., of the form $O(q^{\frac{r+1}{2}})$), we need (all components of) this Adams power to not have any invariants when regarded as representations of $\pi_1(\GG_m)$. When doing Weil descent, what we are really looking at is the invariant space of the (Frobenius twisted) $r$-th tensor power of this sheaf, which is a much larger object. In particular, we may get some undesired additional invariants. In this case the monodromy group is semisimple, and therefore its determinant has some finite order $N$. Then its $N$-th tensor power is definitely going to have non-zero invariant space, which (in general) would not be present if we just considered the Adams power. 

On the other hand, for the Kummer case we can write $N_r-q^r$ as a sum of multiplicative character sums, parameterized by the set of all non-trivial multiplicative characters $\chi$ of $k^\star$ of order divisible by $\frac{q-1}{e}$. Even though it is not possible to realize these sums as the Frobenius traces of an $\ell$-adic sheaf on a scheme, recent work of Katz (\cite{mellin}, especially remark 17.7) shows that these sums are approximately distributed like traces of random elements on a compact Lie group. For generic $f$, this group is the unitary group $U_{d-1}$. In particular, all tensor powers of this ``representation'' (the standard $(d-1)$-dimensional representation of $U_{d-1}$) have zero invariant space, and this makes our method work well.

We conjecture that one should get a similar estimate for Kummer hypersurfaces of the form
$$
y^{\frac{q-1}{e}}=f(x_1,\ldots,x_n)
$$
where $f\in k_r[x_1,\ldots,x_n]$ is in some Zariski open set, namely one should have
$$
\left|N_r-q^{nr}\right|\leq C_{n,d,e,r}q^{\frac{nr+1}{2}}
$$
for some $C_{n,d,e,r}$ independent of $q$. However, the conditions in this case should necessarily be more restrictive, as shown by the example
$$
y^{q-1}=x_1x_2+1
$$
in which $f(x_1,x_2)=x_1x_2+1$ is as smooth as it can be but it is easy to check that
$$
N_r=q^{2s}+(q-2)q^r
$$  
for every odd $q$ and every $r$.

The author would like to thank Daqing Wan for pointing out some mistakes in an earlier version of this article.

\section{The Artin-Schreier case}

 Let $k=\Fq$ be a finite field of characteristic $p$, $k_r=\Fqr$ the extension of $k$ degree $r$ inside a fixed algebraic closure $\bar k$, and $f\in k_r[x]$ a polynomial of degree $d$ prime to $p$. Let $A_f$ be the Artin-Schreier curve defined in $\AAA^2_{k_r}$ by the equation
\begin{equation}\label{curve}
 y^q-y=f(x)
\end{equation}
and denote by $N_f$ its number of $k_r$-rational points. The group of $k$-rational points of the affine line $\AAA^1$ acts on $A_f(k_r)$ by $\lambda\cdot(x,y)=(x,y+\lambda)$.

By the general Artin-Schreier theory, an element $z\in k_r$ can be written as $y^q-y$ for some $y\in k_r$ if and only if $\Tr_{k_r/k}(z)=0$, and in that case there are exactly $q$ such $y$'s. Therefore
$$
N_f=q\cdot\#\{x\in k_r|\Tr(f(x))=0\}
$$
where $\Tr=\Tr_{k_r/k}$ is the trace map $k_r\to k$.

 Let us recall the Weil descent setup (cf. for instance \cite{soto-andrade}). Fix an basis ${\mathcal B}=\{\alpha_1,\ldots,\alpha_r\}\subseteq k_r$ of $k_r$ over $k$, and consider the polynomial $S(x_1,\ldots,x_r)=\sum_{j=1}^r f^{\sigma^j}(\sigma^j(\alpha_1)x_1+\cdots+\sigma^j(\alpha_{r})x_{r})\in k_r[x_1,\ldots,x_r]$, where $\sigma\in \Gal(k_r/k)$ is the Frobenius automorphism and $f^{\sigma^j}$ means applying $\sigma^j$ to the coefficients of $f$. Since the coefficients of $S$ are invariant under the action of $\Gal(k_r/k)$, $S\in k[x_1,\ldots,x_r]$. 

  Let $V$ be the subscheme of $\AAA^r_k$ defined by the polynomial $S$. Notice that a point $(x_1,\ldots,x_{r})\in k^r$ is in $V(k)$ if and only if $\sum_{j=1}^r f^{\sigma^j}(\sigma^j(\alpha_1)x_1+\cdots+\sigma^j(\alpha_{r})x_{r})=\sum_{j=1}^r\sigma^j(f(\alpha_1x_1+\cdots+\alpha_rx_{r}))=0$, if and only if $\Tr(f(\alpha_1x_1+\cdots+\alpha_rx_{r}))=0$. Since $\{\alpha_1,\ldots,\alpha_{r}\}$ is a basis of $k_r$ over $k$, we conclude that 
\begin{equation}\label{nftrace}
 N_f=q\cdot\#\{x\in k_r|\Tr(f(x))=0\}=q\cdot\# V(k).
\end{equation}

On the other hand, $V\otimes k_r$ is isomorphic, under a linear change of variable, to the hypersurface defined by $f^{\sigma}(x_1)+f^{\sigma^2}(x_2)+\cdots+f^{\sigma^r}(x_{r})=0$ in $\AAA^r_{k_r}$. Since $d$ is prime to $p$, $V$ has at worst isolated singularities, and its projective closure has no singularities at infinity. In particular, we get:

\begin{thm}\label{AScurve} Let $f\in k_r[x]$ be a polynomial of degree $d$ prime to $p$. If the hypersurface defined in $\AAA^r_k$ by $f^\sigma(x_1)+f^{\sigma^2}(x_2)+\cdots+f^{\sigma^r}(x_{r})=0$ is non-singular, the number $N_f$ of $k_r$-rational points on $C_f$ satisfies the estimate
 $$
|N_f-q^r|\leq \frac{(d-1)^{r+1}-(-1)^{r}(d-1)}{d}q^{\frac{r+1}{2}}+\frac{(d-1)^{r}-(-1)^{r-1}(d-1)}{d}q^{\frac{r}{2}}\leq
$$
$$
\leq (d-1)^{r}q^{\frac{r+1}{2}}.
$$
\end{thm}

\begin{proof}
 If $\bar V$ is the projective closure of $V$ in $\PP^r_k$ and $V_0=\bar V-V$, we have
$$
\# V(k)-q^{r-1}=\#\bar V(k)-\#V_0(k)-(\#\PP^{r-1}(k)-\#\PP^{r-2}(k))=
$$
$$
=(\#\bar V(k)-\#\PP^{r-1}(k))-(\# V_0(k)-\#\PP^{r-2}(k))
$$
so
$$
|N_r-q^r|=q\cdot |\# V(k)-q^{r-1}|\leq 
$$
$$
\leq q\cdot (|\#\bar V(k)-\#\PP^{r-1}(k)|+|\# V_0(k)-\#\PP^{r-2}(k)|)\leq
$$
$$
\leq \frac{(d-1)^{r+1}-(-1)^{r}(d-1)}{d}q^{\frac{r+1}{2}}+\frac{(d-1)^{r}-(-1)^{r-1}(d-1)}{d}q^{\frac{r}{2}}
$$
since $\bar V$ and $V_0$ are non-singular of degree $d$ and dimension $r-1$ and $r-2$ respectively.
\end{proof}

As noted in \cite[end of section 3]{artin-schreier}, the non-singularity condition is generic in every linear space of polynomials of degree $d$ that contains the constants: if $\lambda\in k_r$ is such that $\Tr_{k_r/k}(\lambda)$ is not a critical point of $f^\sigma(x_1)+\cdots+f^{\sigma^r}(x_r)$, then $f-\lambda$ satisfies the condition. The order of magnitude of the constant is polynomial in $d$ of degree $r$, essentially the same as in \cite{artin-schreier}. However, the leading coefficient there decreases rapidly with $r$, whereas here it is always $1$.
 
The same procedure can be applied to Artin-Schreier hypersurfaces. Let $f\in k_r[x_1,\ldots,x_n]$ be a Deligne polynomial, that is, its degree $d$ is prime to $p$ and its highest homogeneous form defines a non-singular projective hypersurface. Let $B_f$ be the Artin-Schreier hypersurface defined in $\AAA^{n+1}_{k_r}$ by the equation
\begin{equation}
 y^q-y=f(x_1,\ldots,x_n)
\end{equation}
and denote by $N_f$ its number of $k_r$-rational points. Like in the previous case, we have
$$
N_f=q\cdot\#\{(x_1,\ldots,x_n)\in k_r^n|\Tr(f(x_1,\ldots,x_n))=0\}
$$
where $\Tr$ is the trace map $k_r\to k$. Let $S\in k_r[\{x_{i,j}|1\leq i\leq n,1\leq j\leq r\}]$ be the polynomial
$$
\sum_{j=1}^r f^{\sigma^j}\left(\sum_{i=1}^{r}\sigma^j(\alpha_i)x_{1,i},\ldots,\sum_{i=1}^{r}\sigma^j(\alpha_i)x_{n,i}\right)
$$
which has coefficients in $k$, and $V$ the subscheme of $\AAA^{nr}_k$ defined by $S$. Again $N_f=q\cdot \# V(k)$, and $V\otimes k_r$ is isomorphic to the hypersurface defined by $f^\sigma(x_{1,1},\ldots,x_{n,1})+\cdots+f^{\sigma^r}(x_{1,r},\ldots,x_{n,r})=0$. Since this hypersurface is non-singular at infinity, we get

\begin{thm}
\label{AShypersurface} Let $f\in k_r[x_1,\ldots,x_n]$ be a Deligne polynomial of degree $d$ prime to $p$. If the hypersurface defined in $\AAA^{nr}_k$ by $f^\sigma(x_{1,1},\ldots,x_{n,1})+\cdots+f^{\sigma^r}(x_{1,r},\ldots,x_{n,r})=0$ is non-singular, the number $N_f$ of $k_r$-rational points on $C_f$ satisfies the estimate
 $$
|N_r-q^{nr}|\leq \frac{(d-1)^{nr+1}-(-1)^{nr}(d-1)}{d}q^{\frac{nr+1}{2}}+\frac{(d-1)^{nr}-(-1)^{nr-1}(d-1)}{d}q^{\frac{nr}{2}}\leq
$$
$$
\leq (d-1)^{nr}q^{\frac{nr+1}{2}}.
$$
\end{thm}

\section{The Kummer case}

Fix a positive integer $e$ which divides $q-1$. Let $E_f$ be the Kummer curve defined in $\AAA^2_{k_r}$ by the equation
\begin{equation}\label{curve2}
 y^{\frac{q-1}{e}}=f(x)
\end{equation}
and denote by $N_f$ its number of $k_r$-rational points. The group $k^\star$ of $k$-rational points of the torus $\GG_m$ acts on $E_f(k_r)$ by $\lambda\cdot(x,y)=(x,\lambda^ey)$.

A non-zero element $z\in k_r$ can be written as $y^{\frac{q-1}{e}}$ for some $y\in k_r$ if and only if $\mathrm{N}_{k_r/k}(z)^e=1$, and in that case there are exactly $\frac{q-1}{e}$ such $y$'s. Therefore
$$
N_f=\#Z(k_r)+\frac{q-1}{e}\cdot\#\{x\in k_r|\mathrm{N}(f(x))^e=1\}
$$
where $\mathrm{N}$ is the norm map $k_r\to k$ and $Z$ is the subscheme of $\AAA^1_{k_r}$ defined by $f=0$. 

If we apply the Weil descent method to identify the set $\{x\in k_r|\mathrm{N}(f(x))^e=1\}$ with the set of $k$-rational points on a scheme over $k$ like we did in the Artin-Schreier case we get a scheme geometrically isomorphic to the one defined by $(f^\sigma(x_1)\cdots f^{\sigma^r}(x_{r}))^e=1$, which is highly singular at infinity. In particular, its higher cohomology groups do not vanish. However, these cohomology groups are relatively easy to control as we will see.  

Fix an basis ${\mathcal B}=\{\alpha_1,\ldots,\alpha_r\}\subseteq k_r$ of $k_r$ over $k$, and consider the polynomial $T(x_1,\ldots,x_{r})=\prod_{j=1}^r f^{\sigma^j}(\sigma^j(\alpha_1)x_1+\cdots+\sigma^j(\alpha_r)x_{r})\in k_r[x_1,\ldots,x_r]$, where $\sigma\in\Gal(k_r/k)$ is the Frobenius automorphism and $f^{\sigma^j}$ means applying $\sigma^j$ to the coefficients of $f$. The coefficients of $T$ are invariant under the action of $\Gal(k_r/k)$, so $T\in k[x_1,\ldots,x_r]$. 

  For any $\lambda\in k$, let $W_\lambda$ be the subscheme of $\AAA^r_k$ defined by $T=\lambda$. A point $(x_1,\ldots,x_{r})\in k^r$ is in $W_\lambda(k)$ if and only if $\prod_{j=1}^r f^{\sigma^j}(\sigma^j(\alpha_1)x_1+\cdots+\sigma^j(\alpha_r)x_{r})=\prod_{j=1}^r\sigma^j(f(\alpha_1x_1+\cdots+\alpha_rx_{r}))=\lambda$, if and only if $\mathrm{N}(f(\alpha_1x_1+\cdots+\alpha_rx_r))=\lambda$. Since $\{\alpha_1,\ldots,\alpha_r\}$ is a basis of $k_r$ over $k$, we conclude that 
\begin{equation}\label{nfnorm}
 N_f=\#Z(k_r)+\frac{q-1}{e}\cdot\#\{x\in k_r|\mathrm{N}(f(x))^e=1\}= \end{equation}
$$
 = \#Z(k_r)+\frac{q-1}{e}\sum_{\lambda^e=1}\#\{x\in k_r|\mathrm{N}(f(x))=\lambda\}=\#Z(k_r)+\frac{q-1}{e}\sum_{\lambda^e=1}\# W_\lambda(k).
$$

Now $W_\lambda\otimes k_r$ is isomorphic, under a linear change of variables, to the hypersurface defined by $f^\sigma(x_1)f^{\sigma^2}(x_2)\cdots f^{\sigma^r}(x_{r})=\lambda$. This hypersurface is highly singular at infinity, so in general we are not going to obtain good bounds for its number of rational points. For instance, in the simplest case $f(x)=x$, the hypersurface is a product of $r-1$ tori. In particular, it has non-zero cohomology with compact support in all degrees between $r-1$ and $2r-2$.

In order to understand the cohomology of these hypersurfaces, it will be convenient to consider the entire family $f^\sigma(x_1)\cdots f^{\sigma^r}(x_{r})=\lambda$ parameterized by $\lambda$ and study the relative cohomology sheaves. We will do this in a more general setting. Let $f_1,\ldots,f_s\in k_r[x]$ be polynomials of degree $d$, and let $F_s:\AAA^s_{k_r}\to\AAA^1_{k_r}$ be the map defined by $F_s(x_1,\ldots,x_{s})=f_1(x_1)\cdots f_s(x_{s})$. Fix a prime $\ell\neq p$ and an isomorphism $\iota:\QQ\to\CC$, and let $K_s:=\R {F_s}_!\QQ\in\Dbc(\AAA^1_{k_r},\QQ)$ be the relative $\ell$-adic cohomology complex with compact support of $F_s$. For dimension reasons, $\HHH^j(K_s)=0$ for $j<0$ and $j>2s-2$.

\begin{lem}
 Suppose that $f_i$ is square-free for every $i=1,\ldots,s$. Then
\begin{enumerate}
\item $\HHH^j(K_s)_{|\GG_m}=0$ for $j<s-1$.
\item If $s\geq 2$, $\HHH^{2s-2}(K_s)_{|\GG_m}$ is the Tate-twisted constant sheaf $\QQ(1-s)$.
 \item $\HHH^j(K_s)_{|\GG_m}$ is geometrically constant of weight $2(j-s+1)$ for $s\leq j\leq 2s-3$.
 \item $\HHH^{s-1}(K_s)_{|\GG_m}$ contains a subsheaf ${\mathcal F}_s$ which is the extension by direct image of a smooth sheaf on an open subset $V\hookrightarrow\GG_m$ of rank $s(d-1)^s$, pure of weight $s-1$, unipotent at $0$ and totally ramified at infinity, such that the quotient $\HHH^{s-1}(K_s)_{|\GG_m}/{\mathcal F}_s$ is geometrically constant of rank $d^s-(d-1)^s$ and weight $0$.
 \item $\HH^1_c(\GG_m,\FF_s)$ is pure of weight $0$ and dimension $(d-1)^s$.
\end{enumerate}

If all $f_i$ split completely in $k_r$ one can replace ``geometrically constant'' by ``Tate-twisted constant'' everywhere and $\Gal(\bar k_r/k_r)$ acts trivially on $\HH^1_c(\GG_m,\FF_s)$.

\end{lem}

\begin{proof}
  We will proceed by induction on $s$, as in \cite[Th\'eor\`eme 7.8]{kloosterman}. For $s=1$, (1), (2) and (3) are empty, so we only need to prove (4) and (5). In this case, $K_1={f_1}_!\QQ[0]$. There is a natural trace map ${f_1}_!\QQ\to\QQ$, let ${\mathcal F}_1$ be its kernel. Since $d$ is prime to $p$, the inertia group $I_\infty$ at infinity acts on ${\mathcal F}_1$ via the direct sum of all its non-trivial characters of order divisible by $d$. In particular, ${\mathcal F}_1$ is totally ramified at infinity, and is clearly pure of weight $0$. Now from the exact sequence $0\to\FF_1\to {f_1}_!\QQ\to \QQ\to 0$ we get $\HH^1_c(\GG_m,\FF_1)\hookrightarrow\HH^1_c(\GG_m,f_!\QQ)=\HH^1_c(U_1,\QQ)=\HH^0_c(Z_1,\QQ)$ which is pure of weight $0$, where $Z_1\subseteq\AAA^1$ is the subscheme defined by $f_1=0$ and $U_1=\AAA^1-Z_1$. Moreover, $\dim\HH^1_c(\GG_m,\FF_1)=\dim\HH^1_c(\GG_m,{f_1}_!\QQ)-\dim\HH^1_c(\GG_m,\QQ)=\dim\HH^1_c(U_1,\QQ)-\dim\HH^1_c(\GG_m,\QQ)=d-1$ since $f_1$ has $d$ distinct roots in $\bar k$. If $f_1$ splits completely in $k_r$, then $U_1(k_r)=U_1(\bar k_r)$ and therefore $\Gal(\bar k_r/k_r)$ acts trivially on $\HH^1_c(U_1,\QQ)$ and a fortiori on $\HH^1_c(\GG_m,\FF_1)$. 

From now on let us denote $K(f_1)=K_1$ and $\FF(f_1)=\FF_1$ in order to keep track of the polynomial from which they arise. We move now to the induction step, so suppose the lemma has been proved for $s-1$. Since $F_s$ is the composition of $F_{s-1}\times f_s$ and the multiplication map $\mu:\AAA^2_{k_r}\to\AAA^1_{k_r}$, we get $K_s=\R\mu_!(\AAA^ 1\times\AAA^1,K_{s-1}\boxtimes K(f_s))$. In particular, ${K_s}_{|\GG_m}=\R\mu_!(\GG_m\times\GG_m,K_{s-1}\boxtimes K(f_s))$. From the distinguished triangles
$$
\FF(f_s)[0]\to K(f_s)\to\QQ[0]\to
$$
and
$$
\FF_{s-1}[2-s]\to K_{s-1}\to L_{s-1}\to
$$
where $L_{s-1}$ is the ``constant part'' of $K_{s-1}$, we get the distinguished triangles
\begin{equation}\label{DT1}
\R\mu_!(K_{s-1}\boxtimes\FF(f_s)[0])\to {K_s}_{|\GG_m} \to \R\mu_!(\pi_1^\star K_{s-1})\to,
\end{equation}
\begin{equation}\label{DT2}
\R\mu_!(\pi_1^\star \FF_{s-1})[2-s]\to\R\mu_!(\pi_1^\star K_{s-1})\to\R\mu_!(\pi_1^\star L_{s-1})\to
\end{equation}
and
\begin{equation}\label{DT3}
\R\mu_!(\FF_{s-1}\boxtimes\FF(f_s))[2-s]\to\R\mu_!(K_{s-1}\boxtimes\FF(f_s)[0])
\to\R\mu_!(L_{s-1}\boxtimes\FF(f_s)[0])\to
\end{equation}
where $\pi_1,\pi_2:\GG_m\times\GG_m\to\GG_m$ are the projections.

Let $\sigma:\GG_m\times\GG_m\to\GG_m\times\GG_m$ be the automorphism given by $(u,v)\mapsto (u,uv)$. Then $\mu=\pi_2\circ\sigma$ and $\pi_1=\pi_1\circ\sigma$, so
$$
\R\mu_!(\pi_1^\star \FF_{s-1})=\R{\pi_2}_!(\pi_1^\star\FF_{s-1})=\R\Gamma_c(\GG_m,\FF_{s-1})
$$
where the last object is seen as a geometrically constant object (in fact constant if $f_1,\ldots,f_{s-1}$ split in $k_r$) in $\Dbc(\GG_m,\QQ)$. By part (4) of the induction hypothesis, we have $\HH^1_c(\GG_m,\FF_{s-1})=0$ for $i=0,2$, so $\R\Gamma_c(\GG_m,\FF_{s-1})[2-s]=\HH^1_c(\GG_m,\FF_{s-1})[1-s]$. Similarly, using the automorphism $(u,v)\mapsto (uv,v)$ we get
$$
\R\mu_!(L_{s-1}\boxtimes \FF(f_s))=\R\Gamma_c(\GG_m,L_{s-1}\otimes\FF(f_s))
$$
and
$$
\R\mu_!(\pi_1^\star L_{s-1})=\R\Gamma_c(\GG_m,L_{s-1})
$$
which are both geometrically constant (and constant if $f_1,\ldots,f_s$ split in $k_r$).

 With these ingredients we can now start proving the lemma. We have already seen that $\R\Gamma_c(\GG_m,\FF_{s-1})[2-s]$ only has non-zero cohomology in degree $s-1$. By induction, $L_{s-1}$ only has non-zero cohomology in degrees $\geq s-2$. Since a constant object has obviously no punctual sections in $\GG_m$, we deduce that $\R\Gamma_c(\GG_m,L_{s-1}\otimes\FF(f_s))$ and $\R\Gamma_c(\GG_m,L_{s-1})$ only have cohomology in degrees $\geq s-1$.

 For the first term in the triangle (\ref{DT3}) we have
$$
\R\mu_!(\FF_{s-1}\boxtimes\FF(f_s))=\R{\pi_2}_!((\pi_1\circ\sigma^{-1})^\star\FF_{s-1}\otimes(\pi_2\circ\sigma^{-1})^\star\FF(f_s))=
$$
$$
=\R{\pi_2}_!(\pi_1^\star\FF_{s-1}\otimes(\pi_2\circ\sigma^{-1})^\star\FF(f_s))
$$
Its fibre over a geometric point $t\in\GG_m$ is $\R\Gamma_c(\GG_m,\FF_{s-1}\otimes\sigma_t^\star\FF(f_s))$, where $\sigma_t(u)=t/u$ is an automorphism of $\GG_m$. Since $\FF_{s-1}\otimes\sigma_t^\star\FF(f_s)$ has no punctual sections, it does not have cohomology in degree $0$, and therefore $\R\mu_!(\FF_{s-1}\boxtimes\FF(f_s))[2-s]$ only has cohomology in degrees $\geq s-1$. Using the distinguished triangles \ref{DT1}, \ref{DT2} and \ref{DT3} this proves (1).

Since $\FF_{s-1}$ is totally ramified at infinity, $\HH^2_c(\GG_m,\FF_{s-1})=0$, so $\R\mu_!(\pi_i^\star \FF_{s-1})[2-s]=\R\Gamma_c(\GG_m,\FF_{s-1})[2-s]$ has no cohomology in degree $\geq s$ (and in particular in degree $2s-2$). On the other hand, since $\FF(f_s)$ is totally ramified at infinity, so are all cohomology sheaves of $L_{s-1}\otimes\FF(f_s)$. Since $L_{s-1}$ only has cohomology in degrees $\leq 2s-4$, the spectral sequence $\HH^i_c(\GG_m,\HHH^j(L_{s-1})\otimes\FF(f_s))\Rightarrow\HH^{i+j}_c(\GG_m,L_{s-1}\otimes\FF(f_s))$ implies that $L_{s-1}\otimes\FF(f_s)$ only has non-zero cohomology in degrees $\leq 2s-3$. Finally, since $\FF(f_s)$ is smooth at $0$ (because $f_s$ is square-free and therefore \'etale over $0$), $\sigma_t^\star\FF(f_s)$ is unramified at infinity and therefore $\FF_{s-1}\otimes\sigma_t^\star\FF(f_s)$ is totally ramified at infinity. In particular, $\HH^2_c(\GG_m,\FF_{s-1}\otimes\sigma_t^\star\FF(f_s))=0$ and $\R\mu_!(\FF_{s-1}\boxtimes\FF(f_s))[2-s]$ has no cohomology in degree $\geq s$ (in particular in degree $2s-2$). From the triangles \ref{DT1} and \ref{DT2} we get then isomorphisms
$$
\HHH^{2s-2}({K_2}_{|\GG_m})\cong \R^{2s-2}\mu_!(\pi_1^\star K_{s-1})\cong\R^{2s-2}\mu_!(\pi_1^\star L_{s-1})\cong
$$
$$
\cong\HH^2_c(\GG_m,\HHH^{2s-4}(L_{s-1}))\cong\HH^2_c(\GG_m,\QQ(2-s))=\QQ(1-s)
$$
by the induction hypothesis and the spectral sequence $\HH^i_c(\GG_m,\HHH^{j}(L_{s-1}))\Rightarrow\HH^{i+j}_c(\GG_m,L_{s-1})$, where the last two objects are regarded as constant sheaves on $\GG_m$. This proves (2).

For (3), we have already seen that the left hand side of triangle \ref{DT3} only has cohomology in degree $s-1$. Similarly, the left hand side of triangle \ref{DT2} $\R\Gamma_c(\GG_m,\FF_{s-1})[2-s]=\HH^1_c(\GG_m,\FF_{s-1})[1-s]$ only has cohomology in degree $s-1$. Since the other two ends of \ref{DT2} and \ref{DT3} are geometrically constant, we conclude that $\HHH^j(K)_{|\GG_m}$ is geometrically constant for $j\geq s$ using triangle \ref{DT1}.

 Let $s\leq j\leq 2s-3$. For any geometrically constant object $L$, we have $\R\Gamma_c(\GG_m,L)=L\otimes\R\Gamma_c(\GG_m,\QQ)\cong L[-1]\oplus L(-1)[-2]$. In particular 
$$\HHH^j(\R\Gamma_c(\GG_m,L_{s-1}))\cong \HHH^{j-1}(L_{s-1})\oplus\HHH^{j-2}(L_{s-1})(-1)$$ is pure of weight $2(j-s+1)$ by induction. Similarly $\HHH^j(\R\Gamma_c(\GG_m,L_{s-1}\otimes\FF(f_s)))=\HHH^j(L_{s-1}\otimes\R\Gamma_c(\GG_m,\FF(f_s)))\cong\HHH^{j-1}(L_{s-1})\otimes\HH^1_c(\GG_m,\FF(f_s))$ is pure of weight $2(j-s+1)$ since $\HH^1_c(\GG_m,\FF(f_s))$ is pure of weight $0$. Using triangle \ref{DT1} this proves that $\HHH^j(K)_{|\GG_m}$ is pure of weight $2(j-s+1)$. 

From triangles \ref{DT1} and \ref{DT3} we get exact sequences
\begin{equation}\label{seq1}
0\to\R^{s-1}\mu_!(K_{s-1}\boxtimes\FF(f_s)[0])\to \HHH^{s-1}({K_s}_{|\GG_m}) \to
\end{equation}
$$
\to \R^{s-1}\mu_!(\pi_1^\star K_{s-1})\to \R^s\mu_!(K_{s-1}\boxtimes\FF(f_s)[0])
$$
and
$$
0\to\R^{1}\mu_!(\FF_{s-1}\boxtimes\FF(f_s))\to\R^{s-1}\mu_!(K_{s-1}\boxtimes\FF(f_s)[0])
\to
$$
$$
\to\HH^{s-1}_c(\GG_m,L_{s-1}\otimes\FF(f_s))\to 0.
$$

We have already shown that $\R^s\mu_!(K_{s-1}\boxtimes\FF(f_s)[0])$ is pure of weight $2(s-s+1)=2$. On the other hand, from triangle \ref{DT2} we get an exact sequence
$$
\HH^1_c(\GG_m,\FF_{s-1})\to\R^{s-1}\mu_!(\pi_1^\star K_{s-1})\to\HHH^{s-1}(\R\Gamma_c(\GG_m,L_{s-1}))
$$
where the left hand side has weight $0$ by part (5) of the induction hypothesis and the right hand side $\HHH^{s-1}(\R\Gamma_c(\GG_m,L_{s-1}))\cong \HHH^{s-1}(L_{s-1}[-1]\oplus L_{s-1}(-1)[-2])=\HHH^{s-2}(L_{s-1})\oplus \HHH^{s-3}(L_{s-1})(-1)=\HHH^{s-2}(L_{s-1})$ also has weight $0$ by part (4) of the induction hypothesis. Therefore $\R^{s-1}\mu_!(\pi_1^\star K_{s-1})$ is pure of weight $0$, and the last arrow in sequence (\ref{seq1}) is trivial:
$$
0\to\R^{s-1}\mu_!(K_{s-1}\boxtimes\FF(f_s)[0])\to \HHH^{s-1}({K_s}_{|\GG_m}) \to
\R^{s-1}\mu_!(\pi_1^\star K_{s-1})\to 0
$$

Let $\FF_s:=\R^1\mu_!(\FF_{s-1}\boxtimes\FF(f_s))$ (the multiplicative convolution of $\FF(f_1),\ldots,\FF(f_s)$). Then $\FF_s\hookrightarrow \HHH^{s-1}({K_s}_{|\GG_m})$, and the quotient sits inside an exact sequence
$$
0\to\HH^{s-1}_c(\GG_m,L_{s-1}\otimes\FF(f_s))\to\HHH^{s-1}({K_s}_{|\GG_m})/\FF_s\to \R^{s-1}\mu_!(\pi_1^\star K_{s-1})\to 0
$$
whose extremes are already known to be geometrically constant by triangle \ref{DT2}. The rank of this quotient is
$$\dim\HH^{s-1}_c(\GG_m,L_{s-1}\otimes\FF(f_s))+\dim\R^{s-1}\mu_!(\pi_1^\star K_{s-1})
$$
$$
=(\dim\HHH^{s-2}(L_{s-1}))(\dim\HH^1_c(\GG_m,\FF(f_s)))+\dim\HH^1_c(\GG_m,\FF_{s-1})+\dim\HH^{s-1}_c(\GG_m,L_{s-1})$$
$$
=(d^{s-1}-(d-1)^{s-1})(d-1)+(d-1)^{s-1}+\dim\HHH^{s-2}(L_{s-1})+\dim\HHH^{s-3}(L_{s-1})
$$
$$
=d^s-d^{s-1}-(d-1)^s+(d-1)^{s-1}+(d^{s-1}-(d-1)^{s-1})
$$
$$
=d^s-(d-1)^s
$$
by parts (4) and (5) of the induction hypothesis.

By \cite[Corollary 6 and Proposition 9]{degeneration}, $\HHH^{s-1}(K_s)$ (and in particular its subsheaf $\FF_s$) does not have punctual sections in $\AAA^1$. Let $j_0:\GG_m\hookrightarrow\AAA^1$ be the inclusion. We claim that $\HH^1_c(\AAA^1,{j_0}_\star\FF_s)=0$. This will prove both that $\FF_s$ is the extension by direct image of its restriction to any open set $j_V:V\hookrightarrow\GG_m$ on which it is smooth and that it is totally ramified at infinity, since from the exact sequences
$$
0\to {j_0}_\star\FF_s\to {j_0}_\star{j_V}_\star j_V^\star\FF_s\to {\mathcal Q}:={j_V}_\star j_V^\star\FF_s/\FF_s\mbox{(punctual)} \to 0
$$
and
$$
0\to {j_\infty}_!{j_0}_\star\FF_s\to {j_\infty}_\star{j_0}_\star\FF_s\to \FF_s^{I_\infty}\to 0
$$
where $j_\infty:\AAA^1\hookrightarrow\PP^1$ is the inclusion, we get injections ${\mathcal Q}\hookrightarrow\HH^1_c(\AAA^1,{j_0}_\star\FF_s)$ and $\FF_s^{I_\infty}\hookrightarrow\HH^1_c(\AAA^1,{j_0}_\star\FF_s)$.

 Let $i_0:\{0\}\hookrightarrow\AAA^1$ be the inclusion. From the exact sequence
$$
0\to {j_0}_!\FF_s\to {j_0}_\star\FF_s\to {i_0}_\star i_0^\star{j_0}_\star\FF_s\to 0
$$
and the fact that $\FF_s$ has no punctual sections we get
$$
0\to \FF_s^{I_0} \to\HH^1_c(\GG_m,\FF_s)\to\HH^1_c(\AAA^1,{j_0}_\star\FF_s)\to 0
$$
where $\FF_s^{I_0}$ is the invariant space of $\FF_s$ as a representation of the inertia group $I_0$. So it suffices to show that $\dim\FF_s^{I_0}\geq \dim\HH^1_c(\GG_m,\FF_s)$ (and then we will automatically have equality). By definition of $\FF_s$, $\HH^1_c(\GG_m,\FF_s)=\HH^2_c(\GG_m\times\GG_m,\FF_{s-1}\boxtimes\FF(f_s))=\HH^1_c(\GG_m,\FF_{s-1})\times\HH^1_c(\GG_m,\FF(f_s))$. Therefore $\HH^1_c(\GG_m,\FF_s)$ is pure of weight $0$ and dimension $(d-1)^{s-1}(d-1)=(d-1)^s$ by induction, thus proving (5). If $f_1,\ldots,f_s$ split in $k_r$ then $\HH^1_c(\GG_m,\FF_s)$ is a trivial $\Gal(\bar k_r/k_r)$-module, also by induction.

 On the other hand, $\HHH^{s-1}(K_s)_{|\GG_m}$ contains $\FF_s$ plus a geometrically constant part of dimension $d^s-(d-1)^s$. So $\dim\HHH^{s-1}(K_s)^{I_0}=\dim\FF_s^{I_0}+(d^s-(d-1)^s)$. Since $\HHH^{s-1}(K_s)$ has no punctual sections, there is an injection $\HHH^{s-1}(K_s)_0\hookrightarrow \HHH^{s-1}(K_s)^{I_0}$, so $\dim\HHH^{s-1}(K_s)^{I_0}\geq \dim\HHH^{s-1}(K_s)_0$. By base change, $\HHH^{s-1}(K_s)_0=\HH^{s-1}_c(\{f_1(x_1)\cdots f_s(x_s)=0\},\QQ)=\HH^{s}_c(\{f_1(x_1)\cdots f_s(x_s)\neq 0\},\QQ)=\HH^s_c(U_1\times\cdots\times U_s,\QQ)=\HH^1_c(U_1,\QQ)\times\cdots\times\HH^1_c(U_s,\QQ)$, where $U_i\subseteq\AAA^1$ is the open set defined by $f_i(x)\neq 0$ (since the $U_i$ only have non-zero cohomology in degrees $1$ and $2$), so $\dim\HHH^{s-1}(K_s)_0=d^s$. We conclude that $\dim\FF_s^{I_0}=\dim\HHH^{s-1}(K_s)^{I_0}-(d^s-(d-1)^s)\geq \dim\HHH^{s-1}(K_s)_0-(d^s-(d-1)^s)=(d-1)^s=\dim\HH^1_c(\GG_m,\FF_s)$.

To prove (4) it only remains to show that ${\FF_s}_{|V}$ is pure of weight $s-1$ and rank $s(d-1)^s$ and has unipotent monodromy action at $0$. Let $t\in \GG_m$ be a geometric point which is not the product of a non-smoothness point of $\FF_{s-1}$ and a non-smoothness point of $\FF(f_s)$. The fibre of $\FF_s$ over $t$ is $\HH^1_c(\GG_m,\FF_{s-1}\otimes\sigma_t^\star\FF(f_s))$. By the choice of $t$, at every point of $\GG_m$ at least one of $\FF_{s-1}$, $\sigma_t^\star\FF(f_s)$ is smooth. Therefore if $\FF_{s-1}\otimes\sigma_t^\star\FF(f_s)$ is smooth in the open set $j_W:W\hookrightarrow \GG_m$, ${j_W}_\star j_W^\star(\FF_{s-1}\otimes\sigma_t^\star\FF(f_s))=({j_W}_\star j_W^\star\FF_{s-1})\otimes({j_W}_\star j_W^\star\sigma_t^\star\FF(f_s))=\FF_{s-1}\otimes\sigma_t^\star\FF(f_s)$. Given that $\FF_{s-1}$ (respectively $\sigma_t^\star\FF(f_s)$) is pure of weight $s-2$, unipotent at $0$ and totally ramified at $\infty$ (resp. pure of weight $0$, unramified at $\infty$ and totally ramified at $0$), $\FF_{s-1}\otimes\sigma_t^\star\FF(f_s)$ is pure of weight $s-2$ and totally ramified at both $0$ and $\infty$, so $\HH^1_c(\GG_m,\FF_{s-1}\otimes\sigma_t^\star\FF(f_s))=\HH^1(\PP^1,{j_\infty}_\star j_W^\star(\FF_{s-1}\otimes\sigma_t^\star\FF(f_s)))$ is pure of weight $s-1$, where $j_\infty:W\hookrightarrow \PP^1$ is the inclusion.

As for the rank, since $\FF_{s-1}\otimes\sigma_t^\star\FF(f_s)$ has no punctual sections and is totally ramified at $0$ and $\infty$, $\dim\HH^1_c(\GG_m,\FF_{s-1}\otimes\sigma_t^\star\FF(f_s))=-\chi(\GG_m,\FF_{s-1}\otimes\sigma_t^\star\FF(f_s))$. By the Ogg-Shafarevic formula, for each of $\FF_{s-1}$, $\sigma_t^\star\FF(f_s)$ its Euler characteristic is ($-1$ times) a sum of local terms for the points of $\PP^1$ where they are ramified. The local terms at $0$, $\infty$ are the Swan conductors, which get multiplied by $D$ upon tensoring with a unipotent sheaf of rank $D$. The local terms corresponding to ramified points in $\GG_m$ (Swan conductor plus drop of the rank) are multiplied by $D$ upon tensoring with an unramified sheaf of rank $D$. Since at every point of $\GG_m$ at least one of $\FF_{s-1}$, $\sigma_t^\star\FF(f_s)$ is unramified, we conclude that $-\chi(\GG_m,\FF_{s-1}\otimes\sigma_t^\star\FF(f_s))=-(d-1)\chi(\GG_m,\FF_{s-1})-(s-1)(d-1)^{s-1}\chi(\GG_m,\FF(f_s))=(d-1)(d-1)^{s-1}+(s-1)(d-1)^{s-1}(d-1)=s(d-1)^s$.

Finally, since $\FF_s^{I_0}\cong\HH^1_c(\GG_m,\FF_s)$ has weight $0$ and $\FF_s$ is pure of weight $s-1$, for every Frobenius eigenvalue of $\FF_s^{I_0}$ there is a unipotent Jordan block of size $s$ for the monodromy of $\FF_s$ at $0$ by \cite[Section 1.8]{weil2}. Since its rank is $s(d-1)^s$, these Jordan blocks fill up the entire space, and therefore the $I_0$ action is unipotent. This finishes the proof of (4) and of the lemma.  
\end{proof}

Now let $T:\AAA^r_{k_r}\to\AAA^1_{k_r}$ be the map defined by the polynomial $T$, and $K'_r:=\R T_!\QQ\in\Dbc(\AAA^1_{k},\QQ)$. After extending scalars to $k_r$, $K'_r$ becomes isomorphic to $K_r$ for $f_j=f^{\sigma^j}$, $j=1,\ldots,r$. Since the results of the lemma are invariant under finite extension of scalars, they also hold for $K'_r$. In particular, for every $r-1\leq j\leq 2r-2$ there exist $\beta_{j,1},\ldots,\beta_{j,d_j}\in\CC$ of absolute value $1$, where $d_j=\mathrm{rank}\;\HHH^j(K_r)$ (or the rank of the constant part if $j=r-1$) such that for every finite extension $k_m$ of $k$ of degree $m$ and every $\lambda\in k_m^\star$
\begin{equation}\label{beta}
\#\{(x_1,\ldots,x_r)\in k_m^r|T(x_1,\ldots,x_r)=\lambda\}=
\end{equation}
$$
=\sum_{j=r-1}^{2r-2} (-1)^j\sum_{l=1}^{d_j}q^{m(j-r+1)}\beta_{j,l}^m+(-1)^{r-1}\Tr(\mathrm{Frob}_{k_m,\lambda}|\FF_r).
$$ 
Taking the sum over all $\lambda\in k_m^\star$ and using the trace formula:
$$
\#\{(x_1,\ldots,x_{r})\in k_m^r|T(x_1,\ldots,x_{r})\neq 0\}=
$$
$$
=\sum_{j=r-1}^{2r-2} (-1)^j\sum_{l=1}^{d_j}(q^m-1)q^{m(j-r+1)}\beta_{j,l}^m+(-1)^{r-1}\sum_{\lambda\in k_m}\Tr(\mathrm{Frob}_{k_m,\lambda}|\FF_r)=
$$ 
$$
=\sum_{j=r-1}^{2r-2} (-1)^j\sum_{l=1}^{d_j}(q^m-1)q^{m(j-r+1)}\beta_{j,l}^m+(-1)^r\Tr(\mathrm{Frob}_{k_m}|\HH^1_c(\GG_m,\FF_r)).
$$ 

Let $b$ be the degree of a splitting field of $f$ over $k_r$. Then the (geometrically constant) cohomology sheaves of $K'_r$ become constant after extending scalars to $k_{br}$. In particular all $\beta_{j,l}$ are $br$-th roots of unity. If $m$ is any positive integer congruent to $1$ modulo $br$ we have then  
$$
\#\{(x_1,\ldots,x_{r})\in k_m^r|T(x_1,\ldots,x_{r})\neq 0\}=
$$
$$
=\sum_{j=r-1}^{2r-2} (-1)^j\sum_{l=1}^{d_j}(q^m-1)q^{m(j-r+1)}\beta_{j,l}+(-1)^r\Tr(\mathrm{Frob}_{k_m}|\HH^1_c(\GG_m,\FF_r))=
$$
$$
=\left(\sum_{l=1}^{d_{2r-2}}\beta_{2r-2,l}\right)q^{mr}+\sum_{j=r-1}^{2r-2}(-1)^{j-1}\left(\sum_{l=1}^{d_{j-1}}\beta_{j-1,l}+\sum_{l=1}^{d_j}\beta_{j,l}\right)q^{m(j-r+1)}+
$$
$$
+(-1)^r\Tr(\mathrm{Frob}_{k_m}|\HH^1_c(\GG_m,\FF_r)).
$$ 

Since $m$ is prime to $r$, $\mathcal B$ is a basis of $k_{mr}$ over $k_{m}$, and thus $T(x_1,\ldots,x_{r})=N_{k_{mr}/k_m}(f(\alpha_ 1x_1+\cdots+\alpha_{r}x_{r}))$. Therefore
$$
\#\{(x_1,\ldots,x_r)\in k_m^r|T(x_1,\ldots,x_r)\neq 0\}=
$$
$$
=\#\{x\in k_{mr}|\mathrm{N}_{k_{mr}/k_m}(f(x))\neq 0\}=
\#\{x\in k_{mr}|f(x)\neq 0\}
$$
and in particular
$$
\left|\#\{(x_1,\ldots,x_r)\in k_m^r|T(x_1,\ldots,x_r)\neq 0\}-q^{mr}\right|\leq d.
$$

Substituting in the formula above, we get
$$
\left|\left(\sum_{l=1}^{d_{2r-2}}\beta_{2r-2,l}-1\right)q^{mr}+\sum_{j=r-1}^{2r-2}(-1)^{j-1}\left(\sum_{l=1}^{d_{j-1}}\beta_{j-1,l}+\sum_{l=1}^{d_j}\beta_{j,l}\right)q^{m(j-r+1)}+\right.
$$
$$
\left.+(-1)^r\Tr(\mathrm{Frob}_{k_m}|\HH^1_c(\GG_m,\FF_r))\right|\leq d
$$ 

Letting $m\to\infty$ and using that $|\Tr(\mathrm{Frob}_{k_m}|\HH^1_c(\GG_m,\FF_r))|\leq (d-1)^r$ is bounded by a constant, we conclude that $\sum_{l=1}^{d_{2r-2}}\beta_{2r-2,l}=1$ and 
$$
\sum_{l=1}^{d_{j-1}}\beta_{j-1,l}+\sum_{l=1}^{d_{j}}\beta_{j,l}=0
$$
for every $r\leq j\leq 2r-2$, so $\sum_{l=1}^{d_{j}}\beta_{j,l}=(-1)^j$ for every $r-1\leq j\leq 2r-2$.

\begin{thm}\label{SEcurve} Let $f\in k_r[x]$ be a square-free polynomial of degree $d$ prime to $p$ and $e|q-1$. Then the number $N_f$ of $k_r$-rational points on the curve
$$
y^{\frac{q-1}{e}}=f(x)
$$
satisfies the estimate
 $$
|N_f-q^r-\delta+1|\leq r(d-1)^{r}(q-1)q^{\frac{r-1}{2}}
$$
where $0\leq\delta\leq d$ is the number of roots of $f$ in $k_r$.
\end{thm}

\begin{proof}
 Substituting the computed values for $\sum_{l=1}^{d_j}\beta_{j,l}$ in equation \ref{beta} for $m=1$ we get
$$
\# W_\lambda(k)=\#\{(x_1,\ldots,x_{r})\in k^r|T(x_1,\ldots,x_{r})=\lambda\}=
$$
$$
=\sum_{j=r-1}^{2r-2} q^{j-r+1}+(-1)^{r-1}\Tr(\mathrm{Frob}_{k,\lambda}|\FF_r)
=\sum_{j=0}^{r-1} q^{j}+(-1)^{r-1}\Tr(\mathrm{Frob}_{k,\lambda}|\FF_r).
$$
 
So, by equation \ref{nfnorm}, we have
$$
N_f=\#Z(k_r)+\frac{q-1}{e}\sum_{\lambda^e=1}\# W_\lambda(k)=
$$
$$
=\delta+\frac{q-1}{e}\sum_{\lambda^e=1}\left(\sum_{j=0}^{r-1} q^j +(-1)^{r-1}\Tr(\mathrm{Frob}_{k,\lambda}|\FF_r)\right)=
$$
$$
=\delta+(q^r-1)+(-1)^{r-1}\frac{q-1}{e}\sum_{\lambda^e=1}\Tr(\mathrm{Frob}_{k,\lambda}|\FF_r),
$$
so
$$
|N_f-q^r-\delta+1|\leq \frac{q-1}{e}\sum_{\lambda^e=1}r(d-1)^rq^\frac{r-1}{2}=r(d-1)^r(q-1)q^\frac{r-1}{2}
$$
since $\FF_r$ is pure of weight $r-1$ and generic rank $r(d-1)^r$ by the lemma, and its rank can only drop at ramified points.
\end{proof}

\begin{rem}\emph{ The condition that $f$ is square-free is necessary, as shown by the example
 $$
y^{q-1}=x^d
$$
in which 
$$
N_r=1+\#\{x\in k_r|N_{k_r/k}(x^d)=1\}=1+\sum_{t\in k,t^d=1}\#\{x\in k_r|N_{k_r/k}(x)=t\}=
$$
$$
=1+\mu_d\cdot(q^{s-1}+q^{r-2}+\cdots+q+1)
$$
where $\mu_d\geq 1$ is the number of $d$-th roots of unity in $k$.}
\end{rem}

\end{document}